\documentclass[10pt,a4,leqno]{amsart}
\usepackage{amsthm}
\usepackage{amsmath}
\usepackage{amssymb}
\usepackage{graphicx}
\usepackage[cmtip,arrow]{xy}
\usepackage{pb-diagram,pb-xy}
\usepackage{bbm}
\usepackage{color}
\allowdisplaybreaks

\makeatletter

\@addtoreset{equation}{section}
\makeatother

\newtheoremstyle{my theoremstyle}
{0.8em}                    
    {0.8em}                    
    {\itshape}                   
    {}                           
    {\scshape}                   
    {.}                          
    {.5em}                       
    {}  

\newtheoremstyle{dfn}
{0.8em}                    
    {0.8em}                    
    {}                   
    {}                           
    {\scshape}                   
    {.}                          
    {.5em}                       
    {}  
    
\theoremstyle{my theoremstyle}
   \newtheorem{thm}{Theorem}[section]
   \newtheorem{lem}[thm]{Lemma}
   \newtheorem{prop}[thm]{Proposition}
   \newtheorem{cor}[thm]{Corollary}
\theoremstyle{dfn}
   \newtheorem{dfn}[thm]{Definition}
   
\theoremstyle{remark}   
   \newtheorem{exa}[thm]{{\scshape Example}}
   \newtheorem{rem}[thm]{{\scshape Remark}}

\newcommand{\dsum}{\displaystyle \sum}
\renewcommand{\1}{\mathbbm{1}}
\newcommand{\C}{\mathbb{C}}
\newcommand{\Q}{\mathbb{Q}}
\newcommand{\Z}{\mathbb{Z}}

\newcommand{\R}{\mathbb{R}}
\newcommand{\F}{\mathbb{F}}
\renewcommand{\P}{\mathbb{P}}
\renewcommand{\a}{\alpha}
\renewcommand{\b}{\beta}
\renewcommand{\c}{\gamma}
\newcommand{\e}{\varepsilon}
\renewcommand{\d}{\delta}
\newcommand{\p}{\varphi}

\newcommand{\ol}[1]{\overline{#1}}

\newcommand{\hF}[5]{{}_{#1}F_{#2}\Bigg({#3\atop#4};#5\Bigg)}
\newcommand{\hFred}[3]{F_{\rm red}\Bigg({#1\atop#2};#3\Bigg)}

\newcommand{\aff}{{\rm aff}}

\begin{document}
\title[Diagonal hypersurfaces and hypergeometric functions]{Artin $L$-functions of diagonal hypersurfaces and generalized hypergeometric functions over finite fields}
\author{Akio Nakagawa}
\date{\today}
\address{Department of Mathematics and Informatics, Graduate school of Science, Chiba university, Inage, Chiba, 263-8522, Japan}
\email{akio.nakagawa.math@icloud.com}
\maketitle
\begin{abstract}
We compute the Artin $L$-function of a diagonal hypersurface $D_\lambda$ over a finite field associated to a character of a finite group acting on $D_\lambda$, and under some condition, express it in terms of hypergeometric functions and Jacobi sums over the finite field. As an application, we derive certain relations among hypergeometric functions over different finite fields from the Grothendieck-Lefschetz trace formula.
\end{abstract}

\section{Introduction}
Let $D_\lambda$ be a diagonal hypersurface over a finite field $\F_q$ defined by the homogeneous equation
\begin{equation*}
X_1^d+\cdots+X_n^d=d\lambda X_1^{h_1}\cdots X_n^{h_n},
\end{equation*}
where $d,n\geq2$, $h_i\geq1$, $\sum_i h_i=d$, gcd$(d,h_1,\dots,h_n)=1$ and $\lambda\in\F_q$. 
Assume that $d\mid q-1$ and let $\mu_d\subset\F_q^\times$ be the group of $d$th roots of unity. Then a subquotient $G$ of $(\mu_d)^n$ acts on $D_\lambda$.
The $l$-adic \'{e}tale cohomology decomposes into $\chi$-eigenspaces 
$$H^*(\ol{D_\lambda};\ol{\Q_l})=\bigoplus_{\chi\in\widehat{G}}H^*(\ol{D_\lambda};\ol{\Q_l})(\chi),$$
where $\widehat{G}=\operatorname{Hom}(G,\ol{\Q}^\times)$ is the character group. Accordingly, the zeta function decomposes into the Artin $L$-functions as 
$$Z(D_\lambda,t)=\prod_{\chi\in\widehat{G}} L(D_\lambda,\chi;t).$$
The first aim of this paper is to describe $L(D_\lambda,\chi;t)$ in terms of hypergeometric functions over finite fields.

Recall that a classical generalized hypergeometric function ${}_mF_n$ over the complex numbers is defined by a power series
\begin{equation*}
\hF{m}{n}{a_1,\dots,a_m}{b_1,\dots,b_n}{z}:=\sum_{k=0}^\infty \dfrac{(a_1)_k\cdots(a_m)_k}{(1)_k(b_1)_k\cdots(b_n)_k}z^k.
\end{equation*}
Here, $(a)_k=\Gamma(a+k)/\Gamma(a)$ where $\Gamma(s)$ is the gamma function and the parameters $a_i, b_j$ are complex numbers with $b_j\not\in \Z_{\leq0}$. Hypergeometric functions over finite fields are defined independently by Koblitz \cite{N.koblitz}, Greene \cite{JG} and McCarthy \cite{Mc} for $m=n+1$, and by Katz \cite{Katz} and Otsubo \cite{Otsubo} in general. We use Otsubo's definition which coincides with McCarthy's one when $m=n+1$ (see Definition \ref{def of F}). Such a function is a map from $\F_q$ to $\ol{\Q}$ whose parameters are characters of $\F_q^\times$, and the definition relies on the analogy between the gamma function and the Gauss sum.

The Artin $L$-function $L(D_\lambda,\chi;t)$ is a generating function of $N_r(D_\lambda;\chi)$ ($r\geq1$), the number of $\F_{q^r}$-rational points associated to $\chi$. For the subvariety $D_\lambda^*$ of $D_\lambda$ defined by $X_1\cdots X_n\neq0$, Koblitz \cite[(3.2)]{N.koblitz} expressed $N_r(D_\lambda^*;\chi)$ in terms of Jacobi sums (see Proposition \ref{kob}). 
We extend this result to $N_r(D_\lambda;\chi)$ and express them in terms of hypergeometric functions (Theorem \ref{main2} and Corollary \ref{cor of main2}). Our result is a  refinement of Salerno's result \cite[Theorem 4.1]{salerno} for $\#D_\lambda(\F_{q^r})=\sum_\chi N_r(D_\lambda;\chi)$.
There is also a result of Miyatani \cite{Miyatani} for more general hypersurfaces.

When $d=n$ (then $h_i=1$ for all $i$), $D_\lambda$ is called the Dwork hypersuface. In this case, we have a more precise formula. The following is a special case of Theorem \ref{N of dwork 2}.
A character $\chi\in\widehat{G}$ is indexed by an element $w=(w_1,\dots,w_n)\in(\Z/d\Z)^n$. 
For $\a\in\widehat{\F_q^\times}$ and $r\geq 1$, we write $\widetilde{\a}=\a\circ {\rm N}_{\F_{q^r}/\F_q}\in\widehat{\F_{q^r}^\times}$.
Fix a character $\p_d\in \widehat{\F_q^\times}$ of exact order $d$. 
Let $\e\in\widehat{\F_q^\times}$ and $\1\in\widehat{G}$ be the trivial characters. 

\begin{thm} Suppose that $\lambda\in\F_q^\times$ and none of $w_i$ is $0$. Then 
\begin{align*}
&N_r(D_{\lambda};\chi)=(-1)^dj(\p_d^w)^r\hFred{\widetilde{\p}_d^{w_1},\widetilde{\p}_d^{w_2},\dots,\widetilde{\p}_d^{w_d}}{\widetilde\e,\widetilde{\p}_d,\dots,\widetilde{\p}_d^{d-1}}{\lambda^d}_{q^r}
\end{align*}
if $\chi\neq\1$, and if $\chi=\1$ the right-hand side is added by $\dfrac{1-q^{r(d-1)}}{1-q^r}$. Here, $F_{\rm red}$ is the reduced hypergeometric function and $j(\p_d^w)$ is a Jacobi sum (see Definition \ref{red F} and subsection 3.3). 
\end{thm}


By the Grothendieck-Lefschetz trace formula, when $D_\lambda$ is non-singular, $L(D_\lambda,\chi;t)$ is essentially the characteristic polynomial of the $q$-Frobenius $F$ acting on the primitive middle cohomology group $H^{n-2}_{\rm prim}(\ol{D_\lambda};\ol{\Q_l})(\chi)$, and $N_r(D_\lambda;\chi)$ is essentially the trace of $F^r$.
The dimension, say $k$, of the cohomology is determined by Katz (see Lemma \ref{dim of H}). Therefore,  $N_r(D_\lambda;\chi)$ are expressed as polynomials in those for $r=1,\dots,k$.
As a result, we obtain relations among hypergeometric functions over different finite fields. 
We may regard such relations as analogues of the Davenport-Hasse relation for Gauss sums over different finite fields (see Proposition \ref{DH}). 

In the case of Dwork hypersurfaces, such formulas become extremely simple (Theorem \ref{main 4}).
As a special case when $k=2$, we have the following relation among finite field analogues of Gauss hypergeometric functions.
\begin{thm}\label{2F1 relation}
Let $a,b,c\in\Z/d\Z$ satisfy $a,b\not\in\{0, c\}$, $c\neq0$ and 
$$c-a-b=\dfrac{d(d-1)}{2}\mod{d}.$$ 
For $\lambda\in\F_q^\times-\mu_d$ and $r\geq1$, we have 
\begin{align*}
&\hF{}{}{\widetilde{\p}_d^a,\widetilde{\p}_d^b}{\widetilde\e,\widetilde{\p}_d^c}{\lambda^d}_{q^r}\\
&=P_r\left(\hF{}{}{\p_d^a,\p_d^b}{\e,\p_d^c}{\lambda^d}_{q},\dfrac{1}{2}\Big(\hF{}{}{\p_d^a,\p_d^b}{\e,\p_d^c}{\lambda^d}_q^2-\hF{}{}{\widetilde{\p}_d^a,\widetilde{\p}_d^b}{\widetilde\e,\widetilde{\p}_d^c}{\lambda^d}_{q^2}\Big)\right).
\end{align*}
Here, $P_r\in\Z[x,y]$ is the unique polynomial satisfying $P_r(\a+\b,\a\b)=\a^r+\b^r$.
\end{thm}
Moreover, we will closely look at the cases $d=3$ (elliptic curves) and $d=4$ (K3 surfaces). Then $k\leq 2$ and $k\leq 3$ respectively for any $\chi$. We show, however, that only one hypergeometric function over $\F_q$ is sufficient to express those over $\F_{q^r}$, hence to express $N_r(D_\lambda;\chi)$ and $L(D_\lambda,\chi;t)$ (Example \ref{d=3}, Corollary \ref{Z(D4,t)} and Theorem \ref{Main5}).

\section{Preliminaries}

\subsection{Zeta functions and Artin $L$-functions}
In this subsection, we recall the definitions of zeta functions and Artin $L$-functions, and their properties. For more details, see \cite{serre} and \cite{weil}.

Let $\F_q$ be a finite field with $q$ elements of characteristic $p$. Let $\F_{q^r}$ be the degree $r$ extension of $\F_q$ in a fixed algebraic  closure $\ol{\F_q}$ of $\F_q$. Let $V$ be a variety over $\F_q$ and put 
$$N_r(V) = \#V(\F_{q^r})\quad (r\in\Z_{>0}).$$
Then, {\it the zeta function of $V$} is defined by
\begin{equation*}\label{zeta} 
Z(V,t)=\exp\left(\sum_{r=1}^{\infty}\dfrac{N_r(V)}{r}t^r\right)\ \in \Q[[t]].
\end{equation*}

Let $G$ be a finite abelian group, and suppose that $G$ acts on $V$ over $\F_q$. Let $F$ be the $q$-Frobenius acting on $V(\ol{\F_q})$. We write $gF$ for the composition.
For $\chi\in \widehat{G}$ and $r\in\Z_{>0}$, put
\begin{align*}
\Lambda(g^{-1}F^r)&:=\#\{x\in V(\ol{\F_q})\mid g^{-1}F^r(x)=x\},\nonumber\\
N_r(V;\chi)&:=\dfrac{1}{\#G}\sum_{g\in G}\chi(g)\Lambda(g^{-1}F^r)\ \in\ol{\Q}.
\end{align*}
{\it The Artin $L$-function of $V$ associated to $\chi$} is defined by
\begin{equation*}\label{artinl}
L(V,\chi;t)=\exp\left(\sum_{r=1}^{\infty}\dfrac{N_r(V;\chi)}{r}t^r\right)\ \in \ol{\Q}[[t]].
\end{equation*}
Since $N_r(V)=\sum_{\chi\in\widehat{G}}N_r(V;\chi)$,
we have $Z(V,t)=\prod_{\chi \in \widehat{G}} L(V,\chi;t)$.
Let $G_0$ be a finite abelian group which acts on $V$ and suppose that $G\subset G_0$. 
Then, the following holds (cf. \cite[(11)]{serre}):
\begin{equation}\label{N=sum N}
N_r(V;\chi)=\sum_{\substack{\chi_0
\in\widehat{G_0}\\\chi_0|_G=\chi}} N_r(V;\chi_0).
\end{equation}

\subsection{Gauss and Jacobi sums}
For any $\eta\in\widehat{\F_q^\times}=\operatorname{Hom}(\F_q^\times,\ol{\Q}^\times)$, we set $\eta(0)=0$ and put
\begin{align*}
\delta(\eta)=\begin{cases}1&({\rm if\ }\eta=\varepsilon),\\ 0&({\rm if\ }\eta\neq \varepsilon).\end{cases}
\end{align*}
Fix a non-trivial additive character $\psi\in{\rm Hom}(\F_q,\ol{\Q}^\times)$.
For $\eta,\eta_1,\dots,\eta_n  \in \widehat{\F_q^\times}$ with $n\geq1$, define the Gauss sum $g(\eta)$ and the Jacobi sum $j(\eta_1,\dots,\eta_n)$, which are finite field analogues of the gamma and beta functions respectively, by 
\begin{align*}
g(\eta)&=-\sum_{x\in \F_q^\times} \eta(x)\psi(x)\ \in\Q(\mu_{p(q-1)}),\\
j(\eta_1,\dots,\eta_n)&=(-1)^{n-1}\sum_{\substack{x_i\in\F_q^\times\\ x_1+\dots+x_n=1}}\eta_1(x_1)\cdots\eta_n(x_n)\ \in\Q(\mu_{q-1}).
\end{align*}
Since $\sum_{x\in \F_q} \psi(x)=0$ and $\psi(0)=1$, we have
\begin{align*}
g(\varepsilon)=1.
\end{align*}
Define 
\begin{equation*}
g^\circ(\eta):=q^{\d(\eta)}g(\eta).
\end{equation*}
The following identities are well-known (cf. \cite[Proposition 2.2]{Otsubo}). For each $\eta\in\widehat{\F_q^\times}$, 
\begin{equation}\label{Gauss sum thm}
g(\eta)g^\circ(\eta^{-1})=\eta(-1)q,
\end{equation}
and if $\eta\neq\e$ then
\begin{equation}\label{abs.val. of g}
|g(\eta)|=\sqrt{q}.
\end{equation}
For $n\geq1$, 
\begin{equation}\label{Jequ(i)}
j(\underbrace{\varepsilon,\dots,\varepsilon}_{n})=\dfrac{1-(1-q)^n}{q}.
\end{equation}
For $\eta_1,\dots,\eta_n\in\widehat{\F_q^\times}$ with not all $\eta_i=\varepsilon$,
\begin{equation}\label{gJ=G}
j(\eta_1,\dots,\eta_n)=\dfrac{g(\eta_1)\cdots g(\eta_n)}{g^\circ(\eta_1\cdots\eta_n)}.
\end{equation}
In particular, if $\eta_1\cdots\eta_n=\varepsilon$, then by (\ref{Gauss sum thm}),
\begin{equation}\label{Jequ(ii)}
j(\eta_1,\dots,\eta_n)=\eta_n(-1)j(\eta_1,\dots,\eta_{n-1})\ (n\geq2).
\end{equation}
We prepare the following lemma obtained by an easy change of variables.
\begin{lem}\label{genJ} For $\eta_1,\dots,\eta_{n+1}\in\widehat{\F_q^\times}$, we have
\begin{align*}
\sum_{\substack{x_i,y\in\F_q^\times\\x_1+\dots+x_n=y}}&\eta_1(x_1)\cdots\eta_n(x_n)\eta_{n+1}(y)\\
&=\begin{cases}
(-1)^{n}(1-q)j(\eta_1,\dots,\eta_n)&(\eta_1\cdots\eta_{n+1}=\e),\vspace{5pt}\\
0&(\eta_1\cdots\eta_{n+1}\neq\e).
\end{cases}
\end{align*}
\end{lem}

We will use the Davenport-Hasse multiplication formula.
\begin{prop}[{cf. \cite[11.3]{BEW}}]\label{DHMF}  Let $m\in \Z_{>0}$ and suppose that $m\mid q-1$. For any $\eta \in \widehat{\mathbb{F}_q^\times}$, we have
\begin{align*}
\dfrac{\prod_{i=0}^{m-1}g(\p_m^i \eta)}{\prod_{i=0}^{m-1}g(\p_m^i)}\eta(m^{m})=g(\eta^m),
\end{align*}
where $\p_m$ is a character of exact order $m$.
\end{prop}

Let $r\geq1$ be an integer and let ${\rm N}_{\F_{q^r}/\F_q}$ be the norm map. For $\eta\in\widehat{\F_q^\times}$, put $\widetilde{\eta}=\eta\circ{\rm N}_{\F_{q^r}/\F_q}\in\widehat{\F_{q^r}^\times}$. Then, the following is well known as the Davenport-Hasse theorem.
\begin{prop}[cf. \cite{weil}]\label{DH}
For each $r\geq 1$, we have 
\begin{equation*} 
g(\widetilde{\eta})=g(\eta)^r.
\end{equation*}
\end{prop}
\subsection{Hypergeometric functions over finite fields}
In this subsection, we recall hypergeometric functions over finite fields. We follow the definitions of Otsubo \cite{Otsubo}.

For $\a,\nu\in\widehat{\F_q^\times}$, we put
\begin{align*} 
(\a)_\nu&:=\dfrac{g(\a\nu)}{g(\a)},\\(\a)_\nu^\circ &:= \dfrac{g^\circ(\a\nu)}{g^\circ(\a)}=q^{\delta(\a\nu)-\delta(\a)}(\a)_\nu.
\end{align*}
In particular, $(\e)_\nu=g(\nu)$ and $(\a)_\e=(\a)^\circ_\e=1$.
By Proposition \ref{DHMF}, we have, for any $m\mid q-1$,
\begin{equation}\label{DHMF for Poc}
\begin{array}{l}
\displaystyle (\a^m)_{\nu^m}=\prod_{i=0}^{m-1} (\a\p_m^i)_\nu\cdot\nu(m^m),\vspace{5pt}\\
\displaystyle (\a^m)_{\nu^m}^\circ=\prod_{i=0}^{m-1} (\a\p_m^i)_\nu^\circ\cdot\nu(m^m).
\end{array}
\end{equation}

\begin{dfn}\label{def of F} For $\a_1,\dots, \a_m, \b_1,\dots, \b_n \in \widehat{\F_q^\times}$ and $\lambda\in\F_q$, define
\begin{align*}
\hF{}{}{\a_1,\dots,\a_m}{\b_1,\dots,\b_{n}}{\lambda}_q:=\dfrac{1}{1-q}\sum_{\nu\in\widehat{\F_q^\times}}\dfrac{(\a_1)_\nu\cdots(\a_m)_\nu}{(\b_1)_\nu^\circ\cdots(\b_n)_\nu^\circ}\nu(\lambda).
\end{align*}
(We often omit writing $q$ of $F(\cdots)_q$). 
\end{dfn}
We only consider the case when $m=n$ in this paper, and then, the values of $F$ are in $\Q(\mu_{q-1})$ (see \cite[Lemma 2.4 (iii)]{Otsubo}). For comparisons with definitions of Koblitz \cite{N.koblitz}, Greene \cite{JG}, Katz \cite{Katz} and McCarthy \cite{Mc}, see \cite[Remark 2.15]{Otsubo}. As a special case, it is known that
\begin{equation}\label{1F0}
\hF{}{}{\a}{\e}{\lambda}=\a^{-1}(1-\lambda)
\end{equation}
for $\a\neq\e$ and $\lambda\neq0$ (cf. \cite[Corollary 3.4]{Otsubo}).

We define reduced hypergeometric functions.
\begin{dfn}\label{red F}
Let $\a_1,\dots,\a_m, \b_1,\dots,\b_n, \c_1,\dots,\c_l\in\widehat{\F_q^\times}$ and assume that $\{\a_1,\dots,\a_m\}$ and $\{\b_1,\dots,\b_n\}$ have an empty intersection. Then, we put
\begin{equation*}
\hFred{\a_1,\dots,\a_m,\c_1,\dots,\c_l}{\b_1,\dots,\b_n,\c_1,\dots,\c_l}{\lambda}_q=\hF{}{}{\a_1,\dots,\a_m}{\b_1,\dots,\b_n}{\lambda}_q.
\end{equation*}
\end{dfn}

When we reduce a hypergeometric function over finite fields, remainder terms appear as the following. 
\begin{lem}[{\cite[Theorem 3.2]{Otsubo}}]\label{red formula} In the situation of Definition \ref{red F}, suppose that $\c_i\neq\c_j$ for all $1\leq i<j\leq l$. Then, 
\begin{align*}
&\hF{}{}{\a_1,\dots,\a_m,\c_1,\dots,\c_l}{\b_1,\dots,\b_n,\c_1,\dots,\c_l}{\lambda}_q\\
&=q^{\d}\hFred{\a_1,\dots,\a_m,\c_1,\dots,\c_l}{\b_1,\dots,\b_n,\c_1,\dots,\c_l}{\lambda}_q+\dfrac{q^{\d}}{q}\sum_{j=1}^l\dfrac{\prod_{i=1}^m (\a_i)_{\c_j^{-1}}}{\prod_{i=1}^n(\b_i)_{\c_j^{-1}}^\circ}\c_j^{-1}(\lambda). 
\end{align*}
Here, $\d=1$ when $\c_j=\e$ for some $j$ and $\d=0$ otherwise.
\end{lem}
\section{Artin $L$-functions and hypergeometric functions}
\subsection{Diagonal hypersurfaces}
Let $h_i\geq1\ (i=1,\dots,n)$ be integers with $h_1+\dots+h_n=d$ and gcd($d,h_1,\dots,h_n$)=1, and let $\lambda\in \F_q$. We consider the diagonal hypersurface $D_\lambda$ in $\P^{n-1}$ over $\F_q$ defined by the homogeneous equation
\begin{equation}\label{diagonal}
D_\lambda : X_1^d+\dots+X_n^d=d\lambda X_1^{h_1}\cdots X_n^{h_n}.
\end{equation}
Note that $D_{\lambda}$ is non-singular if and only if $\Big(\prod_{i=1}^n h_i^{h_i}\Big)\lambda^d\neq1$. Let $D_\lambda^*$ denote the subvariety of $D_\lambda$ defined by $X_1\cdots X_n\neq0$.

Define a variety $D_\lambda^{\aff}$ in $\mathbbm{A}^n$ over $\F_q$ by
\begin{equation*}
D_\lambda^{\aff}:x_1^d+\dots+x_n^d=d\lambda x_1^{h_1}\cdots x_n^{h_n},
\end{equation*}
and let $D_\lambda^{*\aff}$ denote the subvariety of $D_\lambda^\aff$ defined by $x_1\cdots x_n\neq0$. 

Suppose that $d\mid q-1$, so that $\mu_d\subset\F_q^\times$. Define groups by
\begin{equation*}
\widetilde{G}_0:=\mu_d^n\supset \widetilde{G}:=\{\xi\in\widetilde{G}_0\mid \xi^h=1\}.
\end{equation*}
Here, we write $h=(h_1,\dots,h_n)$ and $\xi^h=\xi_1^{h_1}\cdots\xi_n^{h_n}$ for $\xi=(\xi_1,\dots,\xi_n)$. Let $\Delta\subset\widetilde{G}$ be the diagonal subgroup and define  
\begin{equation*}
G_0:=\widetilde{G}_0/\Delta\supset\ G:=\widetilde{G}/\Delta.
\end{equation*}
Let $\widetilde{G}$ act on $D_\lambda^\aff$ and $D_\lambda^{*\aff}$ over $\F_q$ by
\begin{equation*}
\xi\cdot(x_1,\dots,x_n)=(\xi_1x_1,\dots,\xi_nx_n).
\end{equation*}
This induces an action of $G$ on $D_\lambda$ and $D_\lambda^*$ through the natural map $\mathbb{A}^n\backslash 0\rightarrow \mathbb{P}^{n-1}$. Similarly, $\widetilde{G}_0$ acts on $D_0^\aff$ and this action induces an action of $G_0$ on the Fermat hypersurface $D_0$.

 Fix a generator $\varphi$ of $\widehat{\F_q^\times}$. We have the following commutative diagrams:
\begin{equation*}
\begin{diagram}
\node{\widetilde{G}}\arrow{e,J}\arrow{s,A}\node{\widetilde{G}_0}\arrow{s,A}\\
\node{G}\arrow{e,J}\node{G_0\ ,}
\end{diagram}\ \ \ 
\begin{diagram}
\node{\widehat{\widetilde{G}}}\node{\widehat{\widetilde{G}_0}}\arrow{w,A}\\
\node{\widehat{G}}\arrow{n,L}\node{\widehat{G_0}\ .}\arrow{n,L}\arrow{w,A}
\end{diagram}
\end{equation*}
 From now on, we regard $\widehat{G}$ (resp. $\widehat{G_0}$) as a subgroup of $\widehat{\widetilde{G}}$ (resp. $\widehat{\widetilde{G}_0}$).
We have an isomorphism
\begin{align*}
(\Z/d\Z)^n\overset{\cong}{\longrightarrow}\widehat{\widetilde{G}_0}\ ;\ w=(w_1,\dots,w_n)\longmapsto \chi_0^w,
\end{align*}
where $\chi_0^w(\xi):=\p(\xi^w)$. Put $$\chi^w=\chi_0^w|_{\widetilde{G}}.$$ If we put 
\begin{equation*}
W:=\{w\in(\Z/d\Z)^n\mid w_1+\cdots+w_n=0\},
\end{equation*}
then
\begin{equation*}
w\in W\Longleftrightarrow\chi_0^w\in\widehat{G_0}\Longleftrightarrow\chi^w\in\widehat{G}.
\end{equation*}
Note that, for $w,w'\in (\Z/d\Z)^n$,
\begin{equation*}
\chi^{w}=\chi^{w'}\Longleftrightarrow w-w'= mh \text{ for some }m\in\{0,1,\dots,d-1\}. 
\end{equation*}
Here, note that we regard $mh$ as an element of $(\Z/d\Z)^n$. We write $w\sim w'$ when $\chi^w=\chi^{w'}$. Let $\1\in\widehat{G}$ be the unit character (i.e. $\1=\chi^{w}$ with $w\sim0$).

\subsection{Number of rational points with characters}
Recall that for $w\in W$ and $r\geq1$, 
\begin{align*}
N_r(D_\lambda;\chi^w)&=\dfrac{1}{\#G}\sum_{\xi\in G} \chi^w(\xi)\#\left\{X\in D_\lambda\left(\ol{\F_q}\right)\ \middle|\ \xi^{-1}F^r(X)=X\right\},
\end{align*}
where $F$ is the $q$-Frobenius. Note that $\#G=d^{n-2}$.
Without loss of generality, we only consider the case $r=1$.
For any $m\mid q-1$, put
$$\p_m=\p^{\frac{q-1}{m}},$$
a character of exact order $m$.

For the convenience of the reader, let us give a proof of the following result of Koblitz \cite[(3.2)]{N.koblitz}, which was stated without proof.
\begin{prop}\label{kob}
For $\lambda\neq0$ and $w\in W$, we have
\begin{align*}
N_1(D^*_\lambda;\chi^w)=\dfrac{(-1)^{n}}{1-q}\dsum_{\nu\in\widehat{\F_q^\times}}j\left(\varphi_d^{w_1}\nu^{h_1},\dots,\varphi_d^{w_n}\nu^{h_n}\right)\nu^d(d\lambda).
\end{align*}
\end{prop}
\begin{proof}
Note that if $x\in\{x\in D_{\lambda}^{*\aff}(\ol{\F_q})\mid \xi^{-1}F(x)=ax\}$ for some $a\in\ol{\F_q}^\times$ and $\xi\in\widetilde{G}$, then $a^{-1/(q-1)}x\in \{x\in D_{\lambda}^{*\aff}(\ol{\F_q})\mid \xi^{-1}F(x)=x\}$, and that if $x\in D_\lambda^{*\aff}(\ol{\F_q})$ is fixed by $\xi^{-1}F$ for $\xi\in\widetilde{G}$, then for $c\in\ol{\F_q}^\times$, $cx\in D_\lambda^{*\aff}(\ol{\F_q})$ is fixed by $\xi^{-1}F$ if and only if $c\in\F_q^\times$. By these, we have
\begin{equation}\label{NolD=ND} 
N_1(D^*_\lambda;\chi^w)=\dfrac{1}{q-1}N_1\left(D_\lambda^{*\aff};\chi^{w}\right).
\end{equation}
We have
\begin{align}
&\#\left\{x\in D_\lambda^{*\aff}(\ol{\F_q})\ \middle|\ \xi^{-1}F(x)=x\right\}\label{fixed points Daff}\\
&\hspace{30pt}=\#\left\{x\in (\ol{\F_q}^{\times})^n\ \middle|\ x_1^d+\cdots+x_n^d=d\lambda x^h, x^{q-1}=\xi\right\}\nonumber\\
&\hspace{30pt}=\dfrac{1}{d}\#\left\{x\in(\ol{\F_q}^{\times})^n\ \middle|\ (x_1^d+\cdots+x_n^d)^d=(d\lambda)^d x^{dh}, x^{q-1}=\xi\right\}.\nonumber
\end{align}
Here we used the fact that $\#\left\{x\in(\ol{\F_q}^{\times})^n\ \middle|\ x_1^d+\cdots+x_n^d=cd\lambda x^h, x^{q-1}=\xi\right\}$ does not depend on $c\in\mu_d$ by the assumption gcd$(d,h_1,\dots,h_n)=1$.

If we put $u_i=x_i^d$, then $u_i\in\F_q^\times\Leftrightarrow x_i^{q-1}\in\mu_d$. Thus, the last member of (\ref{fixed points Daff}) is equal to
\begin{align*}
d^{n-1}\#\left\{u\in(\F_q^{\times})^n\ \middle|\ (u_1+\cdots+u_n)^d=(d\lambda)^du^h, u^l=\xi\right\},
\end{align*}
where $l:=(q-1)/d$.

Fix $\xi\in \widetilde{G}$, and define the function $f_\xi:\F_q^\times\rightarrow\Z_{\geq0}$ by 
\begin{equation*}
f_\xi(t):=\#\left\{u\in(\F_q^{\times})^n\ \middle|\ (u_1+\cdots+u_n)^d=tu^h,\ u^l=\xi\right\}.
\end{equation*}
By the discrete Fourier transform on $\F_q^\times$, we have
\begin{align*}
f_\xi(t)&=\dfrac{1}{q-1}\sum_{\nu\in\widehat{\F_q^\times}} \widehat{f}(\nu)\nu(t),
\end{align*}
where
\begin{align*}
\widehat{f}(\nu)&:=\sum_{t\in\F_q^\times}f_\xi(t)\nu^{-1}(t)=\sum_{\substack{u\in(\F_q^\times)^n\\ u^l=\xi}}\nu^{-1}\left(\dfrac{(u_1+\dots+u_n)^d}{u^h}\right).
\end{align*}
Therefore, letting $t=(d\lambda)^d$, we have
\begin{align*}
N_1\left(D_\lambda^{*\aff};\chi^{w}\right)&=\sum_{\xi\in \widetilde{G}}\chi^{w}(\xi)f_\xi(t)\\
&=\dfrac{1}{q-1}\sum_{\xi\in \widetilde{G}}\sum_{\nu\in\widehat{\F_q^\times}}\sum_{\substack{u\in(\F_q^\times)^n\\u^l=\xi}}\chi^{w}(\xi)\nu^{-1}\left(\dfrac{(u_1+\dots+u_n)^d}{u^h}\right)\nu(t)\\
&=\dfrac{1}{q-1}\sum_{\xi}\sum_{\nu}\sum_{\substack{u\\u^l=\xi}}\chi_0^w(u^l)\nu(u^h)\nu^{-d}(u_1+\dots+u_n)\nu(t)\\
&=\dfrac{1}{q-1}\sum_{\nu}\sum_{\substack{u\\ u^l\in \widetilde{G}}}\varphi_d(u^w)\nu(u^h)\nu^{-d}(u_1+\dots+u_n)\nu(t).
\end{align*}
Since
\begin{equation*}
\dfrac{1}{d}\sum_{m=0}^{d-1} \p_d(u^{w+mh})=
\begin{cases}
\p_d(u^w)&(u^{lh}=1),\\ 0&(u^{lh}\neq1),
\end{cases}
\end{equation*}
and $u^{lh}=1\Leftrightarrow u^l\in \widetilde{G}$, $N\left(D_\lambda^{*\aff};\chi^w\right)$ is equal to
\begin{align*}
&\dfrac{1}{d(q-1)}\sum_{\nu\in\widehat{\F_q^\times}}\sum_{m=0}^{d-1}\sum_{u\in(\F_q^\times)^n}\p_d(u^{w+mh})\nu(u^h)\nu^{-d}(u_1+\dots+u_n)\nu(t)\\
&=\dfrac{1}{d(q-1)}\sum_{m=0}^{d-1}\sum_{\nu}\sum_{u\in(\F_q^\times)^n}\p_d(u^w)\p_d^m\nu(u^h)\left(\varphi_d^m\nu\right)^{-d}(u_1+\dots+u_n)\varphi_d^m\nu(t)\\
&=\dfrac{1}{q-1}\sum_\nu\sum_{u\in(\F_q^\times)^n}\varphi_d^{w_1}\nu^{h_1}(u_1)\cdots\varphi_d^{w_n}\nu^{h_n}(u_n)\nu^{-d}(u_1+\dots+u_n)\nu(t)\\
&=(-1)^{n-1}\displaystyle\sum_\nu j\left(\varphi_d^{w_1}\nu^{h_1},\dots,\varphi_d^{w_n}\nu^{h_n}\right)\nu(t).
\end{align*}
The first equality follows by $\p_d^m(t)=1$, the second equality follows by replacing $\p_d^m\nu$ with $\nu$, and the last equality follows from Lemma \ref{genJ} and that $w\in W$. Thus, by (\ref{NolD=ND}), the proposition follows.
\end{proof}

\subsection{Hypergeometric expressions}
For brevity, we write
\begin{equation*}
j(\p_d^w)=j(\p_d^{w_1},\dots,\p_d^{w_n})=\dfrac{g(\p_d^{w_1})\cdots g(\p_d^{w_n})}{q}.
\end{equation*}
From now on, we suppose that $dh_i\mid q-1$ for all $i=1,\dots, n$. For each $w\in W$, put
\begin{align*}
F^w(\lambda):=\hF{}{}{\dots,\varphi_{dh_i}^{w_i},\varphi_{dh_i}^{w_i+d},\dots,\varphi_{dh_i}^{w_i+d(h_i-1)},\dots}{\e,\varphi_d,\dots,\varphi_d^{d-1}}{\left(\prod_{j=1}^nh_j^{h_j}\right)\lambda^d}
\end{align*}
$(i$ runs through $1,\dots,n)$. Here, we understand that $w_i$ means its representative in $\{0,\dots,d-1\}$.
\begin{thm}\label{main1} For any $\lambda\neq0$ and $w\in W$,
\begin{align*}
&N_1(D^*_\lambda;\chi^w)\\
&=\begin{cases}
\dfrac{(-1)^n}{q^{\d(m)}}j\left(\p_d^w\right)^{1-\d(m)}F^w(\lambda)+(-1)^{n-1}\dfrac{(1-q)^{n-1}}{q}&(\text{if\ }w=mh),\vspace{5pt}\\(-1)^nj\left(\p_d^w\right)F^w(\lambda)&(\text{if\ }w\not\sim0),\end{cases}
\end{align*}
where $\d(m)=1$ if $m=0$ and $\d(m)=0$ otherwise.
\end{thm}

\begin{proof}
Since $\sum w_i=0$ and $\sum h_i=d$, we have
\begin{equation*}
\p_d^{w_i}\nu^{h_i}=\e\text{ for all }i\Longleftrightarrow w=mh\text{ for some }m\in\{0,\dots,d-1\}\text{ and }\nu=\p_d^{-m}.
\end{equation*}
Thus, if $w\not\sim0$ then, by (\ref{gJ=G}), 
\begin{equation*}
\sum_\nu j\left(\varphi_d^{w_1}\nu^{h_1},\dots,\varphi_d^{w_n}\nu^{h_n}\right)\nu^d(d\lambda)=\sum_\nu\dfrac{g(\varphi_d^{w_1}\nu^{h_1})\cdots g(\varphi_d^{w_n}\nu^{h_n})}{g^\circ\left(\nu^d\right)}\nu^d(d\lambda),
\end{equation*}
and if $w=mh$ for some $m$ then, by (\ref{Jequ(i)}), (\ref{gJ=G}) and $g^\circ (\e)=q$, 
\begin{align*}
&\sum_\nu j\left(\varphi_d^{w_1}\nu^{h_1},\dots,\varphi_d^{w_n}\nu^{h_n}\right)\nu^d(d\lambda)\\
&=\sum_{\nu\neq\p_d^{-m}}\dfrac{g\left(\varphi_d^{w_1}\nu^{h_1}\right)\cdots g\left(\varphi_d^{w_n}\nu^{h_n}\right)}{g^\circ\left(\nu^d\right)}\nu^d(d\lambda)+\dfrac{1-(1-q)^n}{q}\p_d^{-md}(d\lambda)\\
&=\sum_{\nu}\dfrac{g\left(\varphi_d^{w_1}\nu^{h_1}\right)\cdots g\left(\varphi_d^{w_n}\nu^{h_n}\right)}{g^\circ\left(\nu^d\right)}\nu^d(d\lambda)-\dfrac{(1-q)^n}{q}.
\end{align*}

By (\ref{DHMF for Poc}), we have
\begin{equation*}
g^\circ(\nu^d)=g^\circ(\e)(\e)_{\nu^d}^\circ=q(\varepsilon)_\nu^\circ (\varphi_d)_\nu^\circ \cdots (\varphi_d^{d-1})_\nu^\circ \nu(d^d).
\end{equation*}
For each $i=1,\dots,n$, we have similarly
\begin{align*}
\begin{split}
g\left(\varphi_d^{w_i}\nu^{h_i}\right)&=g\left(\left(\varphi_{dh_i}^{w_i}\nu\right)^{h_i}\right)=g\left(\p_d^{w_i}\right)\left(\p_{dh_i}^{w_ih_i}\right)_{\nu^{h_i}}\\
&=g\left(\varphi_d^{w_i}\right)\left(\varphi_{dh_i}^{w_i}\right)_\nu \left(\varphi_{dh_i}^{w_i+d}\right)_\nu\cdots\left(\varphi_{dh_i}^{w_i+d(h_i-1)}\right)_\nu\nu(h_i^{h_i}).
\end{split}
\end{align*}
Thus, 
\begin{align*}
&\sum_\nu\dfrac{g(\varphi_d^{w_1}\nu^{h_1})\cdots g(\varphi_d^{w_n}\nu^{h_n})}{g^\circ\left(\nu^d\right)}\nu^d(d\lambda)\\
&=q^{-1}\prod_{k=1}^ng\left(\varphi_d^{w_k}\right) \sum_\nu \dfrac{\prod_{i=1}^n\left(\varphi_{dh_i}^{w_i}\right)_\nu \left(\varphi_{dh_i}^{w_i+d}\right)_\nu\cdots\left(\varphi_{dh_i}^{w_i+d(h_i-1)}\right)_\nu}{(\varepsilon)_\nu^\circ (\varphi_d)_\nu^\circ \cdots (\varphi_d^{d-1})_\nu^\circ}\nu\Big(\big(\prod_{j=1}^nh_j^{h_j}\big)\lambda^d\Big)\\
&=\dfrac{(1-q)}{q}\prod_{k=1}^ng\left(\varphi_d^{w_k}\right)\cdot F^w(\lambda).
\end{align*}
Therefore, we obtain the theorem by Proposition \ref{kob} and (\ref{gJ=G}), where note that $mh= 0\Leftrightarrow m=0$ by the assumption that gcd$(d,h_1,\dots,h_n)=1$.
\end{proof}

Next, we consider the case where $\lambda=0$. Note that $\widetilde{G}$ (resp. $G$) acts on $D_0^\aff$ (resp. $D_0$) through the inclusion $\widetilde{G}\hookrightarrow \widetilde{G}_0$ (resp. $G\hookrightarrow G_0$).

The following is proved by Weil \cite{weil2}.
\begin{lem}[cf. {\cite[(2.12)]{N.koblitz}}]\label{ND0} For each $w\in W$,
\begin{equation*}
N_1\left(D_0;\p^w\right)=\left\{\begin{array}{ll}0&(\text{if}\ some\ but\ not\ all\ w_i=0),\vspace{5pt}\\\dfrac{1-q^{n-1}}{1-q}&(\text{if }w=0),\vspace{5pt}\\ (-1)^nj\left(\p_d^w\right)&(\text{if}\ all\ w_i\neq0).\end{array}\right.
\end{equation*}
\end{lem}


\begin{thm}\label{main2}
 For any $\lambda\neq0$ and $w\in W$, we have
\begin{align*}
&N_1(D_\lambda;\chi^w)\\
&=\begin{cases}\dfrac{(-1)^n}{q^{\delta(m)}}j\left(\p_d^w\right)^{1-\delta(m)}F^w(\lambda)+\dfrac{1-q^{n-1}}{1-q}+\dfrac{(-1)^{n-1}}{q}+ C&(\text{if\ }w=mh),\vspace{5pt}\\ (-1)^nj\left(\p_d^w\right)F^w(\lambda)+C&(\text{if\ }w\not\sim0).\end{cases}
\end{align*}
Here, $\d(m)$ is as in Theorem \ref{main1} and
\begin{equation*}
C:=(-1)^{n-1}\dsum_{\substack{w'\sim w\\w_i'=0\ \text{for some but not all }i}}j\left(\p_d^{w'}\right).
\end{equation*}
\end{thm}

\begin{proof}It is trivial that 
\begin{equation*}
N_1(D_\lambda;\chi^w)-N_1(D^*_\lambda;\chi^w)=N_1(D_0;\chi^w)-N_1(D^*_0;\chi^w).
\end{equation*}
By (\ref{N=sum N}), we have, for each $w\in W$, 
\begin{align*}
N_1(D^*_0;\chi^w)&=\sum_{\substack{w'\sim w}}N_1\left(D^*_0;\chi_0^{w'}\right),\\
N_1(D_0;\chi^w)&=\sum_{\substack{w'\sim w}}N_1\left(D_0;\chi_0^{w'}\right).
\end{align*}
Hence, 
\begin{equation}
N_1(D_\lambda;\chi^w)=N_1(D^*_\lambda;\chi^w)+\sum_{\substack{w'\sim w}}\left(N_1\left(D_0;\chi_0^{w'}\right)-N_1\left(D^*_0;\chi_0^{w'}\right)\right).\label{kob3.3}
\end{equation}

Letting $u_i=x_i^d$, $v_i=-u_i/u_n$ and $l=(q-1)/d$, we have, by a similar argument as in the proof of Proposition \ref{kob},
\begin{align*}
&N_1\left(D_0^{*\aff};\chi_0^w\right)\\
&=\dfrac{1}{d^n}\sum_{\xi\in \widetilde{G}_0}\chi_0^w(\xi)\#\left\{x\in(\ol{\F_q}^{\times})^n\ \middle|\ \ x_1^d+\dots+x_n^d=0,\ x^{q-1}=\xi\right\}\\
&=\sum_{\xi\in \widetilde{G}_0}\p(\xi^w)\#\left\{u\in(\F_q^{\times})^n\ \middle|\ u_1+\dots+u_n=0,\ u^l=\xi\right\}\\
&=\sum_{\substack{u\in(\F_q^\times)^n\\u_1+\dots+u_n=0}}\p_d(u^w)\\
&=\varphi_d^{w_n}(-1)\cdot(q-1)(-1)^{n-2}j\left(\varphi_d^{w_1},\dots,\varphi_d^{w_{n-1}}\right),
\end{align*}
where note that $\p_d^{w_1+\cdots+w_n}=\e$ at the last equality. Therefore, by (\ref{Jequ(i)}), (\ref{Jequ(ii)}) and  (\ref{NolD=ND}), 
\begin{align*}
N_1\left(D^*_0;\chi_0^w\right)=\begin{cases}(-1)^nj\left(\p_d^w\right)&({\rm if}\ w\neq0),\vspace{5pt} \\ (-1)^{n}\dfrac{1-(1-q)^{n-1}}{q}&({\rm if}\ w=0).\end{cases}
\end{align*}
By Lemma \ref{ND0}, we have
\begin{align*}
N_1&\left(D_0;\chi_0^w\right)-N_1\left(D^*_0;\chi_0^w\right)\\
&=\begin{cases}0&({\rm if\ }w_i\neq0{\rm\ for\ all\ }i),\vspace{5pt}\\\dfrac{1-q^{n-1}}{1-q}+(-1)^{n-1}\dfrac{1-(1-q)^{n-1}}{q}&({\rm if}\ w=0),\vspace{5pt}\\(-1)^{n-1}j\left(\p_d^w\right)&({\rm otherwise}).\end{cases}
\end{align*}
Thus, by Theorem \ref{main1} and \eqref{kob3.3}, we obtain the result.
\end{proof}

\begin{cor}\label{cor of main2}
Suppose that $w\sim 0$ $($i.e. $w=mh$ for some $m\in\{0,\dots,d-1\}$). Then, for any $\lambda\neq0$,
\begin{equation*}
N_1(D_\lambda;\chi^w)=(-1)^nj(\p_d^w)^{1-\d(m)}F^w({\rm red. }\ \p_d^m; \lambda)+\dfrac{1-q^{n-1}}{1-q}+C,
\end{equation*}
where $F^w({\rm red.}\ \p_d^m; \lambda)$ is the function obtained by removing one $\p_d^m$ from both the numerator and the denominator parameters in $F^w(\lambda)$. Furthermore, if gcd$(d,h_i)=1$ for all $i=1,\dots,n$, then
\begin{equation*}
N_1(D_\lambda;\chi^w)=(-1)^nj(\p_d^w)^{1-\d(m)}F_{\rm red}^w(\lambda)+\dfrac{1-q^{n-1}}{1-q}.
\end{equation*}
\end{cor}

\begin{proof}
Note that $\p_{dh_i}^{w_i}=\p_d^m$ for all $i$. Using Lemma \ref{red formula} partially, we have
\begin{equation*}
F^w(\lambda)=q^{\d(m)}F^w({\rm red.\ }\p_d^m;\lambda)+R,
\end{equation*}
where
\begin{equation*}
R:=\dfrac{q^{\d(m)}}{q}\cdot \dfrac{(\p_d^m)_{\p_d^{-m}}^{n-1}\prod_{i=1}^n  \p_d^{-m}(h_i^{h_i})\prod_{j=1}^{h_i-1}(\p_d^m\p_{h_i}^j)_{\p_d^{-m}}}{\prod_{1\leq i\leq d,\ i\neq m} (\p_d^i)_{\p_d^{-m}}^\circ}.
\end{equation*}
Using \eqref{DHMF for Poc} and \eqref{gJ=G}, we have
\begin{align*}
R&=\dfrac{q^{\d(m)}}{q}\cdot\dfrac{(\p_d^m)_{\p_d^{-m}}^\circ}{(\p_d^m)_{\p_d^{-m}}}\cdot\dfrac{\prod_{i=1}^n (\p_d^{mh_i})_{\p_d^{-mh_i}}}{(\e)_\e^\circ}=\dfrac{1}{\prod_{i=1}^n g(\p_d^{mh_i})}=\dfrac{q^{\d(m)}}{q}j(\p_d^w)^{-1+\d(m)}.
\end{align*}
Thus, by Theorem \ref{main2}, we obtain that
\begin{equation*}
N_1(D_\lambda;\chi^w)=F^w({\rm red. }\ \p_d^m; \lambda)+\dfrac{1-q^{n-1}}{1-q}+C.
\end{equation*}

If ${\rm gcd}(d,h_i)=1$ for all $i$, then $F^w({\rm red.}\ \p_d^m;\lambda)=F_{\rm red}^w(\lambda)$ and $C=0$. Here, note that $\p_{h_i}^j$ is not of order $d$ for all $i=1,\dots,n$ and $j=1,\dots h_i-1$, and that $w\sim0$ implies that $w=0$ or $w_i\neq 0$ for all $i$. Therefore, we have the second equation of the corollary.
\end{proof}
\subsection{Artin $L$-functions}
Fix an integer $r\geq1$. For $\eta\in\widehat{\F_q^\times}$, define a character $\widetilde{\eta}\in\widehat{\F_{q^r}^\times}$ by 
$$\widetilde{\eta}:=\eta\circ {\rm N}_{\F_{q^r}/\F_q}.$$ 
To define hypergeometric functions over $\F_{q^r}$, fix a generator $\varphi'$ of $\widehat{\F_{q^r}^\times}$ such that $\p'|_{\F_q^\times}=\p$. 
Then, we have $\varphi'^{\tfrac{q^r-1}{q-1}}=\widetilde{\varphi}$, and
\begin{equation*}
\varphi'_d:=\varphi'^{\tfrac{q^r-1}{d}}=\widetilde{\varphi_d} 
\end{equation*}
for any $d\mid q-1$. 
For convenience, we also write $\widetilde{\p}_d=\widetilde{\p_d}$. 
Put as before, 
\begin{equation*}
F^{w}_r(\lambda)=\hF{}{}{\dots,\widetilde{\varphi}_{dh_i}^{w_i},\widetilde{\varphi}_{dh_i}^{w_i+d},\dots,\widetilde{\varphi}_{dh_i}^{w_i+d(h_i-1)},\dots}{\widetilde{\e},\widetilde{\varphi}_d,\dots,\widetilde{\varphi}_d^{d-1}}{\bigg(\prod_{j=1}^nh_j^{h_j}\bigg)\lambda^d}_{q^r}.
\end{equation*}
When $w=mh$ for some $m\in\{0,\dots,d-1\}$, define $F^{w}_r({\rm red.}\ \widetilde{\p}_d^m;\lambda)$ similarly as in Corollary \ref{cor of main2}.
\begin{cor}\label{main3}
Suppose that $\lambda\neq0$.
\begin{enumerate}
\item If $w=mh$ for some $m\in\{0,\dots,d-1\}$, then
\begin{align*}
&L(D_\lambda,\chi^w;t)\\
&=\exp\left(\sum_{r=1}^\infty\dfrac{(-1)^n}{q^{r\delta(m)}}j\left(\p_d^w\right)^{r(1-\delta(m))}F^{w}_r({\rm red.}\ \widetilde{\p}_d^m;\lambda)\dfrac{t^r}{r}\right)\prod_{k=0}^{n-2}\dfrac{1}{1-q^kt}\cdot C(t),
\end{align*}
where
\begin{equation*}
C(t):=\prod_{{\substack{w'\sim w\\w_i'=0\ \text{for some but not all }i}}}\left(1-j\left(\p_d^{w'}\right)t\right)^{(-1)^n}.
\end{equation*}

\item If $w\not\sim0$, then
\begin{align*}
L(D_\lambda,\chi^w;t)=\exp\left(\sum_{r=1}^\infty(-1)^nj\left(\p_d^w\right)^{r}F^{w}_r(\lambda)\dfrac{t^r}{r}\right)C(t).
\end{align*}
\end{enumerate}
\end{cor}
\begin{proof}
By applying Theorem \ref{main2} and Corollary \ref{cor of main2} for $\F_{q^r}$ and $\p'$, we obtain the formula for $N_r(D_\lambda;\chi^w)$. For the Jacobi sum, by Proposition \ref{DH}, we have
\begin{equation*}
j\left(\widetilde{\p}_d^{w_1},\dots,\widetilde{\p}_d^{w_n}\right)=j\left(\p_d^{w_1},\dots,\p_d^{w_n}\right)^r.
\end{equation*}
Thus, the theorem follows formally.
\end{proof}

\subsection{Dwork hypersurfaces}
In this subsection, let $D_\lambda$ denote the Dwork hypersurfaces of degree $d$, which is the case when $n=d$ and $h_i=1$ for all $i$ in \eqref{diagonal}. 
In this case, for $w,w'\in W$, $\chi^w=\chi^{w'}$ if and only if $w-w'=(m,\dots,m)$ for some $m\in\Z/d\Z$. 
For $w\in W$, put $\d(w)=1$ if $0\in\{w_1,\dots,w_d\}$ and $\d(w)=0$ otherwise.

\begin{thm}\label{N of dwork 2}Let $r\geq1$ and $\lambda\in\F_q^\times$.
\begin{enumerate}
\item If $w=(m,\dots,m)\in W$ with $m\in\Z/d\Z$ $($i.e. $\chi^w=\1)$, then
\begin{align*}
&N_r(D_{\lambda};\chi^w)=(-1)^dj(\p_d^w)^{(1-\d(w))r}\hFred{\overbrace{\widetilde{\p}_d^{m},\widetilde{\p}_d^{m},\dots,\widetilde{\p}_d^{m}}^d}{\widetilde{\e},\widetilde{\p}_d,\dots,\widetilde{\p}_d^{d-1}}{\lambda^d}_{q^r}+\dfrac{1-q^{r(d-1)}}{1-q^r}.
\end{align*}
\item If $\chi^w\neq\1$, then
\begin{align*}
&N_r(D_{\lambda};\chi^w)=(-1)^d(q^{\d(w)}j(\p_d^w))^r\hFred{\widetilde{\p}_d^{w_1},\widetilde{\p}_d^{w_2},\dots,\widetilde{\p}_d^{w_d}}{\widetilde{\e},\widetilde{\p}_d,\dots,\widetilde{\p}_d^{d-1}}{\lambda^d}_{q^r}.
\end{align*}
In particular, if $w$ is a permutation of $(0,1,\dots,d-1)$ $($this occurs only for odd $d)$,
\begin{equation*}
N_r(D_{\lambda};\chi^w)=\begin{cases}0&(\lambda^d\neq1),\vspace{5pt}\\ (-1)^{\tfrac{r(d^2-1)(q-1)}{8d}}q^{\tfrac{r(d-1)}{2}}&(\lambda^d=1).\end{cases}
\end{equation*}
\end{enumerate}
\end{thm}

\begin{proof}By Proposition \ref{DH}, we only have to prove it for $r=1$.

(i) It follows by Corollary \ref{cor of main2}.

(ii) For $w=(w_1,\dots,w_d)\in W$ and $k\in\Z/d\Z$, put
\begin{align*}
n_{w}(k)=\#\{w_i\mid w_i=k\}.
\end{align*} 
By Theorem \ref{main2}, 
\begin{align}\label{N1}
&N_1(D_{\lambda};\chi^w)=(-1)^dj(\p_d^w)\hF{}{}{\p_d^{w_1},\p_d^{w_2},\dots,\p_d^{w_d}}{\e,\p_d,\dots,\p_d^{d-1}}{\lambda^d}\vspace{5pt}\\&\hspace{150pt}+(-1)^{d-1}\dsum_{\substack{w'\sim w\\w_i'=0{\rm\ for\ some\ }i}}j\left(\p_d^{w'}\right).\nonumber
\end{align}

Put $S=\{w_1,\dots,w_n\}\subset\Z/d\Z$, then by Lemma \ref{red formula}, we have
\begin{align*}
&\hF{}{}{\p_d^{w_1},\p_d^{w_2},\dots,\p_d^{w_d}}{\e,\p_d,\dots,\p_d^{d-1}}{\lambda^d}\\
&=q^{\d(w)}\hFred{\p_d^{w_1},\p_d^{w_2},\dots,\p_d^{w_d}}{\e,\p_d,\dots,\p_d^{d-1}}{\lambda^d}+\sum_{c\in S}\dfrac{\prod_{a\in S}(\p_d^{a})_{\p_d^{-c}}^{n_w(a)-1}}{\prod_{b\not\in S}(\p_d^{b})_{\p_d^{-c}}}.
\end{align*}
Noting that $\prod_{i=1}^d g(\p_d^i)=\prod_{i=1}^d g(\p_d^{i-c})$ for $c\in \Z/d\Z$ and using \eqref{gJ=G}, 
\begin{align*}
\sum_{c\in S}\dfrac{\prod_{a\in S}(\p_d^{a})_{\p_d^{-c}}^{n_w(a)-1}}{\prod_{b\not\in S}(\p_d^{b})_{\p_d^{-c}}}&=\sum_{c\in S}\prod_{i=1}^d\dfrac{(\p_d^{w_i})_{\p_d^{-c}}}{(\p_d^i)_{\p_d^{-c}}}\\
&=\sum_{c\in S}\prod_{i=1}^d \dfrac{g(\p_d^{w_i-c})g(\p_d^i)}{g(\p_d^{w_i})g(\p_d^{i-c})}\\
&=\sum_{c\in S}\prod_{i=1}^d \dfrac{g(\p_d^{w_i-c})}{g(\p_d^{w_i})}\\
&=j(\p_d^w)^{-1}\sum_{c\in S}j(\p_d^{w-c})\\
&=j(\p_d^w)^{-1}\sum_{\substack{w'\sim w\\w_i'=0{\rm\ for\ some\ }i}}j\left(\p_d^{w'}\right).
\end{align*}
Thus, we obtain the anterior half of (ii) by \eqref{N1}.

If $w=(w_1,\dots,w_d)$ is a permutation of $(0,1,\dots,d-1)$, then 
\begin{align*}
\hFred{\p_d^{w_1},\p_d^{w_2},\dots,\p_d^{w_d}}{\e,\p_d,\dots,\p_d^{d-1}}{\lambda^d}&=\dfrac{1}{1-q}\sum_{\nu}\nu(\lambda^d)=\begin{cases}
0&(\lambda^d\neq1),\vspace{5pt}\\ -1&(\lambda^d=1).
\end{cases}
\end{align*}
Thus, we have the latter half of (ii) by the anterior half of (ii) and 
\begin{equation*}
j(\e,\p_d,\dots,\p_d^{d-1})=(-1)^{\tfrac{(d^2-1)(q-1)}{8d}} q^{\tfrac{d-3}{2}},
\end{equation*}
which follows by \eqref{Gauss sum thm} and that $d$ is odd.
\end{proof}

\begin{rem}
For the Dwork hypersurface, McCarthy \cite[Corollary 2.6]{Mc2} expresses  the number $\#D_\lambda(\F_q)$ in terms of his hypergeometric functions.
Theorem \ref{N of dwork 2} is a refinement of this result. 
Goodson (\cite{HG-1} and \cite[Theorem 1.2]{HG-2} for $d=4$ and odd $d$) and Kumabe (\cite{Kumabe} for $d=6$) show similar results in terms of Greene's hypergeometric functions. 
\end{rem}

For each $r\geq1$ and $w\in W$, put
\begin{equation*}
F_{r,\rm red}^{w}(\lambda)=\hFred{\widetilde{\p}_d^{w_1},\widetilde{\p}_d^{w_2},\dots,\widetilde{\p}_d^{w_d}}{\widetilde\e,\widetilde{\p}_d,\dots,\widetilde{\p}_d^{d-1}}{\lambda^d}_{q^r}.
\end{equation*}
Similarly as Corollary \ref{main3}, we have the following corollary of Theorem \ref{N of dwork 2}.

\begin{cor}\label{L of dwork}
For any $\lambda\in\F_q^\times$ and $w\in W$, we have the following:
\begin{enumerate}
\item For $w=(m,\dots,m)\in W$ with $m\in\Z/d\Z$, 
\begin{equation*}
L(D_{\lambda},\chi^w;t)=\dfrac{\exp\Big(\sum_{r=1}^\infty j(\p_d^w)^{r(1-\d(w))}F_{r,\rm red}^{w}(\lambda)\dfrac{t^r}{r}\Big)^{(-1)^d}}{(1-t)(1-qt)\cdots(1-q^{d-2}t)}.
\end{equation*}
\item For $w\in W$ with $\chi^w\neq\1$,
\begin{equation*}
L(D_{\lambda},\chi^w;t)=\exp\left(\sum_{r=1}^\infty q^{r\d(w)}j(\p_d^w)^rF_{r,\rm red}^{w}(\lambda)\dfrac{t^r}{r}\right)^{(-1)^d}.
\end{equation*}
In particular, for $w\in W$ which is a permutation of $(0,1,\dots,d-1)$,
\begin{equation*}
L(D_{\lambda},\chi^w;t)=
\begin{cases}
1&(\lambda^d\neq1),\vspace{5pt}\\ \left(1-(-1)^{(q-1)/d}q^{(d-1)/2}t\right)^{-1}&(\lambda^d=1).
\end{cases}
\end{equation*}
\end{enumerate}
\end{cor}
\begin{rem}
 On the factorization of the zeta functions of Dwork hypersurfaces, Goutet \cite{Goutet} gave a factorization of $Z(D_{\lambda},t)$ considering the actions of the groups $G$ and $\mathfrak{S}_d$.
\end{rem}

\section{Hypergeometric functions over different finite fields}
\subsection{$l$-adic \'{e}tale cohomology }
Let $X$ be a smooth hypersurface in $\P^{n-1}$ over $\F_q$ and write $\ol{X}=X\otimes_{\F_q}\ol{\F_q}$. Let $l$ be a prime number with $l\neq p$. Write $H^m(\ol{X},\ol{\Q_l})$ for the $l$-adic \'{e}tale cohomology of $\ol{X}$. By the weak Lefschetz theorem (cf. \cite[p.106]{FK}), the map
\begin{equation*}
H^m(\ol{\P^{n-1}},\ol{\Q_l})\longrightarrow H^m(\ol{X},\ol{\Q_l})
\end{equation*}
induced by the embedding $X\hookrightarrow \P^{n-1}$ is an isomorphism for $m<n-2$. Thus, by the Poincar\'{e} duality and the structure of the $l$-adic \'{e}tale cohomology of $\P^{n-1}$, we have 
\begin{equation*}
H^{m}(\ol{X},\ol{\Q_l})=
\begin{cases}
0&(m{\rm\ is\ odd,\ }m\neq n-2),\\ \ol{\Q_l}\Big(-\dfrac{m}{2})&(m{\rm \ is\ even,\ } m\neq n-2,\ 0\leq m\leq 2(n-2)).
\end{cases}
\end{equation*}

Let $H\subset\P^{n-1}$ be a hyperplane and $h\in H^2(\ol{X},\ol{\Q_l}(1))$ be the cohomology class associated with $H\cap X$ (independent of $H$). Then, by the hard Lefschetz theorem (cf. \cite[p.274]{FK}), the composition
\begin{equation*}
H^{n-4}(\ol{X},\ol{\Q_l}(-1))\xrightarrow{\cup h}H^{n-2}(\ol{X},\ol{\Q_l})\xrightarrow{\cup h}H^n(\ol{X},\ol{\Q_l}(1))
\end{equation*}
is an isomorphism, where $\cup$ is the cup product. Thus, we have the primitive decomposition
\begin{equation}\label{prim. decom.}
H^{n-2}(\ol{X},\ol{\Q_l})\cong \begin{cases}
H^{n-2}_{\rm prim}(\ol{X},\ol{\Q_l})&(n\mbox{ is odd}),\vspace{5pt}\\
H^{n-2}_{\rm prim}(\ol{X},\ol{\Q_l})\oplus \ol{\Q_l}\Big(-\dfrac{n-2}{2}\Big)&(n\mbox{ is even}),
\end{cases}
\end{equation} 
where
 \begin{equation*}
 H^{n-2}_{\rm prim}(\ol{X},\ol{\Q_l}):=\text{Ker}\left(H^{n-2}(\ol{X},\ol{\Q_l})\xrightarrow{\cup h}H^n(\ol{X},\ol{\Q_l}(1))\right).
\end{equation*}

Now let $X=D_{\lambda}$ ($\lambda\in\F_q$) with $\lambda^d\prod_{i=1}^n h_i^{h_i}\neq1$. Fix an embedding $\ol{\Q}\hookrightarrow \ol{\Q_l}$ and identify $\widehat{G}$ with ${\rm Hom}(G,\ol{\Q_l}^\times)$. Then, we have a decomposition
\begin{equation*}
H^m(\ol{D_\lambda},\ol{\Q_l})=\bigoplus_{\chi\in\widehat{G}}H^m(\ol{D_\lambda},\ol{\Q_l})(\chi).
\end{equation*}
Here, for a vector space $V$ over $\ol{\Q_l}$ equipped with a $G$-action, we write
\begin{equation*}
V(\chi)=\{v\in V\mid g\cdot v=\chi(g)v\mbox{ for all }g\in G\}.
\end{equation*}

Since the $G$-action fixes $h$, it is compatible with the primitive decomposition \eqref{prim. decom.}. Noting that $H^{2k}(\ol{D_{\lambda}},\ol{\Q_l}(k))$ is generated by $h^k$ if $2k\neq n-2$, we have, for $0\leq k\leq n-2$ ($2k\neq n-2$),
\begin{align*}
H^{2k}(\ol{D_\lambda},\ol{\Q_l})(\chi)=\begin{cases}
\ol{\Q_l}(-k)&(\chi=\1),\\
0&(\chi\neq\1),
\end{cases}
\end{align*}
and
\begin{align*}
H^{n-2}(\ol{D_\lambda},\ol{\Q_l})(\chi)=\begin{cases}
H^{n-2}_{\rm prim}(\ol{D_\lambda},\ol{\Q_l})(\chi)\oplus\ol{\Q_l}\Big(-\dfrac{n-2}{2}\Big)&(\chi=\1, n\text{ is even}),\vspace{5pt}\\
H^{n-2}_{\rm prim}(\ol{D_\lambda},\ol{\Q_l})(\chi)&(\text{otherwise}).
\end{cases}
\end{align*}
The dimension of this cohomology group is determined by Katz.
\begin{lem}[{\cite[Lemma 3.1 (i)]{Katz ano}}] \label{dim of H}
For any $w\in W$,
\begin{equation*}
\dim H_{\rm prim}^{n-2}(\ol{D_{\lambda}},\ol{\Q_l})(\chi^w)=\#\{ m\in\Z/d\Z\mid \d(w+mh)=0\},
\end{equation*}
where $\d(w)$ is as in Theorem \ref{N of dwork 2}. 
\end{lem}

\subsection{Relations among hypergeometric functions over different fields}
By the Grothendieck-Lefschetz trace formula (cf. \cite[Theorem 2.9]{FK}), we have
\begin{align}
N_r(D_{\lambda};\chi)&=\sum_{i=0}^{2(n-2)}(-1)^i{\rm Tr}\big((F^{r})^*\mid H^i(\ol{D_\lambda},\ol{\Q}_l)(\chi)\big).\label{Nr Lefschetz}
\end{align}
Therefore, 
\begin{equation*}
L(D_\lambda,\chi;t)=\det\big(1-F^*t\mid H_{\rm prim}^{n-2}(\ol{D_\lambda},\ol{\Q}_l)(\chi)\big)^{(-1)^{n-1}}
\end{equation*}
when $\chi\neq\1$, and the right-hand side is multiplied with $\prod_{i=0}^{n-2} (1-q^it)^{-1}$ when $\chi=\1$. We can write
\begin{equation}
\det\big(1-F^*t\mid H_{\rm prim}^{n-2}(\ol{D_\lambda},\ol{\Q}_l)(\chi)\big)=\prod_{i=1}^k (1-\a_it),\label{det}
\end{equation}
where $k=\dim H_{\rm prim}^{n-2}(\ol{D_{\lambda}},\ol{\Q_l})(\chi)$. By the Weil conjecture proved by Deligne \cite{deligne}, we have $|\a_i|=q^{(n-2)/2}$ for all $i$. 

Let $e_i$ ($0\leq i\leq k)$ be the elementary symmetric polynomial of degree $i$ in indeterminates $\a_1,\dots,\a_k$ and put $$p_r=\sum_{i=1}^k \a_i^r.$$ There exist unique polynomials $P_{r}(x_1,\dots, x_k)\in\Z[x_1,\dots, x_k]$ ($r\geq 1$) satisfying
\begin{equation*}
p_r=P_{r}(e_1,\dots,e_k).\label{sym. poly.}
\end{equation*}
On the other hand, there exist  polynomials $Q_i(x_1,\dots, x_i)\ (1\leq i\leq k)$ such that
\begin{equation*}
e_i=Q_i(p_1,\dots,p_i)\in\Q[x_1,\dots,x_i].
\end{equation*}
This fact follows for example from Newton's identity
\begin{equation*}
ie_i=\sum_{j=1}^i (-1)^{j-1}e_{i-j}p_j\hspace{20pt}(1\leq i\leq k).
\end{equation*}
If we put
\begin{equation*}
R_r(x_1,\dots, x_k)=P_r\big(Q_1(x_1), Q_2(x_1,x_2), \dots, Q_k(x_1,\dots,x_k)\big),
\end{equation*}
then
\begin{equation*}
p_r=R_r(p_1,\dots,p_k)
\end{equation*}
for any $r\geq 1$. 

For $\a_i$ as in \eqref{det}, we have 
\begin{equation}N_r(D_\lambda;\chi)=\begin{cases}
(-1)^np_r&(\chi\neq\1),\vspace{5pt}\\
(-1)^np_r+\dfrac{1-q^{n-1}}{1-q}&(\chi=\1).
\end{cases}\label{N=a}
\end{equation}
Therefore, it follows from Theorem \ref{main2} that we can write $F^w_r(\lambda)$ in terms of $F_1^w(\lambda),\dots,F_k^w(\lambda)$ and Jacobi sums, where $k$ is determined by Lemma \ref{dim of H}.

\subsection{Dwork hypersurfaces}
Let us write down such relations more concretely when $D_\lambda$ ($\lambda^d\neq1$) is the Dwork hypersurface of degree $d$. For $w\in W$, we have  by Lemma \ref{dim of H}
\begin{equation*}
k(w):=\dim H^{d-2}(\ol{D_\lambda},\ol{\Q_l})(\chi^w)=d-\#\{w_i\mid i=1,\dots,d\}.
\end{equation*}
Let $F_{r,\rm red}^w(\lambda)$ be as in subsection 3.5. By Theorem \ref{N of dwork 2} and \eqref{N=a},
\begin{equation}
j'(\p_d^w)^{r}F_{r,\rm red}^w(\lambda)=\sum_{i=1}^{k(w)} \a_i^r,\quad j'(\p_d^w):=\begin{cases}
j(\p_d^w)^{1-\d(w)}&(\chi^w=\1),\vspace{5pt}\\
q^{\d(w)}j(\p_d^w)&(\chi^w\neq\1).
\end{cases}\label{F=Tr}
\end{equation}
Hence we obtain the following.
\begin{thm}\label{main 4}
For any $w\in W$, $\lambda \in \F_q-\mu_d$ and $r\geq 1$, we have
\begin{equation*}
F_{r,\rm red}^w(\lambda)=R_r(F_{1,\rm red}^w(\lambda),\dots,F_{k(w),\rm red}^w(\lambda)).
\end{equation*}
\qed
\end{thm}
We also have the following consequence of the Weil conjecture.
\begin{thm}
For $w\in W$ with $\d(w)=0$ and $\lambda\in\F_{q^r}-\mu_d$, 
\begin{equation*}
\big|F_{r, \rm red}^w(\lambda)\big|\leq k(w).
\end{equation*}
\end{thm}
\begin{proof}
We are reduced to the case when $r=1$. The statement for $\lambda=0$ is clear. By assumption, $|j(\p_d^w)|=q^{(d-2)/2}$ by \eqref{abs.val. of g} and \eqref{gJ=G}. When $\lambda\neq 0$, the theorem follows from \eqref{F=Tr} and that $|\a_i|=q^{(d-2)/2}$ for all $i$.
\end{proof}

Now, let us consider the case when $k(w)=2$. Then Theorem \ref{main 4} becomes
$$F_{r,\rm red}^w(\lambda)=P_r\Big(F_{1,\rm red}^w(\lambda),\dfrac{1}{2}(F_{1,\rm red}^w(\lambda)^2-F_{2,\rm red}^w(\lambda))\Big).$$
As an example, we take $w=(a,b,\check{0},1,\cdots,\check{c},\cdots,d-1)$, where $a,b\not\in\{0,c\}$, $c\neq0$ and 
$$ c-a-b=\dfrac{d(d-1)}{2}\mod{d}.$$
Then
\begin{equation*}
F_{r, \rm red}^w(\lambda)=\hF{}{}{\widetilde{\p}_d^a,\widetilde{\p}_d^b}{\widetilde\e,\widetilde{\p}_d^c}{\lambda^d}_{q^r}.
\end{equation*}
Thus, we obtain Theorem \ref{2F1 relation}.

\begin{exa}\label{d=3}
Let $d=3$ and $\lambda^3\neq 1$. Then $D_{\lambda}$ is the Hesse elliptic curve and 
$$H^1(\ol{D_\lambda},\ol{\Q_l})=H^1_{\rm prim}(\ol{D_\lambda},\ol{\Q_l})(\1)$$
by Lemma \ref{dim of H}.
By Theorem \ref{N of dwork 2} and \eqref{Jequ(ii)}, we have, for $m=1,2$ and $r\geq 1$,
\begin{align}
N_r\left(D_{\lambda}\right)&=N_r(D_\lambda;\1)\label{ND3}\\
&=(-1)^{\tfrac{q-1}{3}mr+1}j\left(\p_3^m,\p_3^m\right)^r\hF{}{}{\widetilde{\p}_3^m,\widetilde{\p}_3^m}{\widetilde\e,\widetilde{\p}_3^{2m}}{\lambda^3}_{q^r}+1+q^r.\nonumber
\end{align}
In this case, $\a_1\a_2=q$ by the Poincar\'{e} duality. Therefore, we have
\begin{equation*}
F_{r,\rm red}^{w}(\lambda)=P_r\left(F_{1,\rm red}^{w}(\lambda),\dfrac{q}{j(\p_3^m,\p_3^m)^2}\right)
\end{equation*}
for any $r\geq1$.

\end{exa}
\begin{rem}\label{MTY}
The complex periods of $D_{\lambda}\ (\lambda\in \C,\ \lambda^3\neq1)$ are computed in \cite[Theorem 1]{MTY}, one of which is
\begin{equation*}
B\left(\dfrac{1}{3},\dfrac{1}{3}\right)\hF{2}{1}{\frac{1}{3},\frac{1}{3}}{\frac{2}{3}}{\lambda^3}-\lambda B\left(\dfrac{2}{3},\dfrac{2}{3}\right)\hF{2}{1}{\frac{2}{3},\frac{2}{3}}{\frac{4}{3}}{\lambda^3}.
\end{equation*}
Our formula \eqref{ND3} can be regarded as a finite field analogue of this result.
\end{rem}

\subsection{Dwork K3 surfaces}
 We consider the case when $D_\lambda$ is the Dwork hypersurfaces of degree $4$. Suppose that $\lambda^4\neq1$, so that $D_{\lambda}$ is a K3 surface, and $\lambda\neq0$. Now, $\widehat{G}$ consists of 16 characters, which are up to permutation of indices $w_1,\dots,w_4$:
 \begin{align*}
& \1,\\
&\chi^{(1,1,3,3)}\ (3\mbox{ characters, note that }\chi^{(1,1,3,3)}=\chi^{(3,3,1,1)}),\\ 
&\chi^{(1,2,2,3)}\ (12\mbox{ characters}). 
\end{align*}
Note that the number $N_r(D_{\lambda};\chi^w)$ does not change under permutations of indices. Recall that
$$k(0,0,0,0)=3, \quad k(1,1,3,3)=2,\quad k(1,2,2,3)=1.$$
\begin{prop} Let $\lambda\in \F_q^\times$ (we do not assume $\lambda \not \in \mu_4$).
\begin{enumerate}
\item For $m\in\{1,2,3\}$,
\begin{align*}
&N_r(D_{\lambda};\1)=\hF{}{}{\widetilde\e,\widetilde\e,\widetilde\e}{\widetilde{\p}_4,\widetilde{\p}_4^2,\widetilde{\p}_4^3}{\lambda^4}_{q^r}+1+q^r+q^{2r}\nonumber\\
&=j(\p_4^m,\p_4^m,\p_4^m,\p_4^m)^r\hFred{\widetilde{\p}_4^m,\widetilde{\p}_4^m,\widetilde{\p}_4^m,\widetilde{\p}_4^m}{\widetilde\e,\widetilde{\p}_4,\widetilde{\p}_4^2,\widetilde{\p}_4^3}{\lambda^4}_{q^r}+1+q^r+q^{2r}.\nonumber
\end{align*}

\item \begin{equation*}
N_r(D_{\lambda};\chi^{(1,1,3,3)})=(\p_2^r(1-\lambda^2)+\p_2^r(1+\lambda^2))q^r.
\end{equation*}

\item \begin{align*}
N_r(D_{\lambda};\chi^{(1,2,2,3)})=(-1)^{\frac{r(q-1)}{4}}\p_2^r(1-\lambda^4)q^r.
\end{align*}
\end{enumerate}

\end{prop}

\begin{proof}
(i) Immediate consequence of Theorem \ref{N of dwork 2}.

(ii) By loc. cit., we have
\begin{align*}
N_r(D_{\lambda;4};\chi^{(1,1,3,3)})=j(\p_4,\p_4,\p_4^3,\p_4^3)^r\hF{}{}{\widetilde{\p}_4,\widetilde{\p}_4^3}{\widetilde\e,\widetilde{\p}_4^2}{\lambda^4}_{q^r}.
\end{align*}
First, $j(\p_4,\p_4,\p_4^3,\p_4^3)=q$ by \eqref{gJ=G} and \eqref{Gauss sum thm}. 
Secondly, using \eqref{DHMF for Poc}, \eqref{Gauss sum thm} and \eqref{gJ=G}, we have for $\nu\in\widehat{\F_{q^r}^\times}$,
\begin{align*}
\dfrac{(\widetilde{\p}_4)_\nu(\widetilde{\p}_4^3)_\nu}{(\widetilde\e)_\nu^\circ (\widetilde{\p}_4^2)_\nu^\circ}
=\dfrac{(\widetilde{\p}_4^2)_{\nu^2}}{(\widetilde\e)_{\nu^2}^\circ}=j(\nu^{-2},\widetilde{\p}_4^2),
\end{align*}
and hence,
\begin{align*}
\hF{}{}{\widetilde{\p}_4,\widetilde{\p}_4^3}{\widetilde\e,\widetilde{\p}_4^2}{\lambda^4}_{q^r}&=\dfrac{1}{1-q^r}\sum_\nu j(\nu^{-2},\widetilde{\p}_4^2)\nu(\lambda^4)\\
&=\dfrac{1}{q^r-1}\sum_\nu\sum_{u\in\F_{q^r}^\times}\nu^{-2}(u)\widetilde{\p}_4^2(1-u)\nu(\lambda^4)\\
&=\dfrac{1}{q^r-1}\sum_u\widetilde{\p}_4^2(1-u)\sum_\nu\nu\Big(\dfrac{\lambda^4}{u^2}\Big)\\
&=\widetilde{\p}_4^2(1-\lambda^2)+\widetilde{\p}_4^2(1+\lambda^2).
\end{align*}
Noting that $\widetilde{\p}_4^2(\lambda)=\p_2(\lambda)^{\frac{1-q^r}{1-q}}=\p_2^r(\lambda),$
the formula follows.

(iii) By Theorem \ref{N of dwork 2}, we have
\begin{align*}
N_r(D_{\lambda};\chi^{(1,2,2,3)})=j(\p_4,\p_4^2,\p_4^2,\p_4^3)^r\hF{}{}{\widetilde{\p}_4^2}{\widetilde\e}{\lambda^4}_{q^r},
\end{align*}
and the formula follows by \eqref{1F0}, \eqref{gJ=G} and \eqref{Gauss sum thm}.
\end{proof}

\begin{cor}\label{Z(D4,t)}
Let $m\in\{1,2,3\}$ and put 
\begin{equation*}
Q(t)=1-p_1t+\frac{p_1^2-p_2}{2}t^2-\frac{p_1^3-3p_1p_2+2p_3}{6}t^3,
\end{equation*}
where
$$p_r=j(\p_d^m,\p_d^m,\p_d^m,\p_d^m)^rF_{r,\rm red}^{(m,m,m,m)}(\lambda).$$
Then
\begin{align*}
&Z(D_{\lambda},t)=\dfrac{1}{(1-t)(1-qt)(1-uqt)^3(1-u'qt)^3(1-vqt)^{12}Q(t)(1-q^2t)},
\end{align*}
where $u=\p_2(1-\lambda^2)$, $u'=\p_2(1+\lambda^2)$ and $v=(-1)^{(q-1)/4}uu'$. 
\end{cor}


In fact, $F_{r,\rm red}^{w}(\lambda)$ hence $Q(t)$ and $Z(D_\lambda,t)$ are determined only by $F_{1,\rm red}^{w}(\lambda)$.
\begin{thm}\label{Main5}Let $P_r(x_1,x_2)$ be the polynomial defined in subsection 4.2 for $k=2$. Let $w=(m,m,m,m)$ ($m\in\{1,2,3\}$). Then, for any $r\geq1$,
\begin{equation*}
F_{r,\rm red}^{w}(\lambda)=\Big(\frac{\p_2(1-\lambda^4)q}{j(\p_4^w)}\Big)^r+P_r\Big(F_{1,\rm red}^{w}(\lambda)-\frac{\p_2(1-\lambda^4)q}{j(\p_4^w)}, \Big(\frac{q}{j(\p_4^w)}\Big)^2 \Big).
\end{equation*}
\end{thm}
\begin{proof}
Since $\det(1-F^*t\mid H^2(\ol{D_\lambda},\Q_l))\in \Q[t]$ (independent of $l$), it follows that $Q(t)\in\Q[t]$. Hence we can write 
$$Q(t)=(1-\a_1t)(1-\a_2t)(1-\a_3t)$$
with $\a_1\in\R$, $\ol{\a_2}=\a_3$ and $|\a_i|=q$ for all $i$. It is known that the highest term of $\det(1-F^*t\mid H^2(\ol{D_\lambda},\Q_l))$ is 
$q^{22}t^{22}$. One can show this fact by reducing to the case of the Fermat quartic surface ($\lambda=0$) similarly as in the proof of Lemma \ref{dim of H}. Hence $\a_1\a_2\a_3=\a_1 q^2=uu'q^3$, i.e. $\a_1=\p_2(1-\lambda^4)q$ (this fact is also proved in \cite{Asakura}). Since 
\begin{equation*}
F_{r,\rm red}^{w}(\lambda)=\dfrac{\a_1^r+\a_2^r+\a_3^r}{j(\p_d^w)^r},\label{NrD4}
\end{equation*}
we obtain the theorem.
\end{proof}

\begin{rem}
When $m=2$, we have $j(\p_4^w)=q$. 
\end{rem}

\section*{Acknowledgments}

This paper basis on the author's master's thesis. I would like to thank my supervisor Noriyuki Otsubo for his constant support and encouragements. I would also like to thank Shigeki Matsuda and Ryojun Ito for helpful discussions and advice.


\end{document}